\documentclass[a4paper,twoside,12pt]{article}
\usepackage[english]{babel}

\usepackage{amsmath,amsfonts,amssymb,amsthm}
\newtheorem{teo}{Theorem}
\newtheorem{lem}{Lemma}

\newtheorem{corol}{Corollary}

 \title{{\bf   Kipriyanov's Fractional Calculus  Prehistory  and Legacy
 }}

\author{Maksim \,V.~Kukushkin   \\ \\
 \small  \textit{Moscow State University of Civil Engineering, 129337,  Moscow, Russia}\\
 \small\textit{Russian Academy of Science, Institute of Applied Mathematics and Automation,   360051,  Nalchik, Russia}\\
\textit{\small\textit{kukushkinmv@rambler.ru}} }

\date{}
\begin{document}

\maketitle

\begin{abstract}
This paper is partly a historical  survey of various  approaches and methods in the fractional calculus, partly a description of the Kipriyanov  extraordinary  theory in comparison with the classical one. The significance and outstanding methods in constructing the  independent  Kipriyanov  fractional calculus theory are convexly stressed, also we represent modern results involving the Kipriyanov operator and corresponding generalization under the semigroup theory point of view.

\end{abstract}
\begin{small}\textbf{Keywords:} Fractional power of the m-accretive  operator; infinitesimal  generator of a semigroup;    Kipriyanov fractional differential operator;
 strictly accretive operator; abstract evolution equations.  
\end{small}

\section{Brief historical review}

\noindent{\bf 1. Birth  of Kipriyanov's fractional calculus }\\

Ivan Alexandrovich Kipriyanov was born on
the Urals in the Chelyabinsk region. In 1945,   Kipriyanov I.A. entered the Faculty of Physics and Mathematics of the Chelyabinsk Pedagogical Institute.
In 1949 he was accepted to graduate school
 Steklov  V.A. Mathematical Institute of the USSR Academy of Sciences, his supervisor was Keldysh M.V.
In 1954   Kupriyanov I.A.  defended  his Ph. D thesis
"On summation of Fourier series
and interpolation processes for
functions of two variables". In this
work,  Kupriyanov I.A. found
a class of functions of two variables,
to which it is possible to completely transfer the results of Lozinsky S.M.  that, in a sense, every theorem on convergence and summation
of a one-dimensional Fourier series can be transferred to convergence
and summation of the corresponding trigonometric interpolation
process with equally spaced nodes. Moreover, these results
remain valid for more general approximating processes.
In his review of the dissertation, Prof.  Lozinsky S.M. noted the successful
choice of a class of functions of two variables. This choice, as
 Lozinsky S.M. wrote, required analytical insight and provided
the success of the work. In his review of the abstract of the dissertation, Prof.  Kantorovich L.V.
noted that, having familiarized himself with the work of  Kipriyanov I.A. on
the abstract and two reports in Leningrad in November 1953 at the seminar
on the theory of functions and functional analysis of Leningrad University,
he formed a very positive opinion about it. The work made
a very favorable impression on the other participants of the seminar,
as its discussion showed. It should be noted that the seminar was attended
by Professors  Kantorovich L.V.,  Lozinsky S.M.,  Natanson I.P., etc.
In the late 50s and early 60s of our century, active
work was carried out on the study of various kinds of functional spaces, important in themselves and also playing an important role in the modern
theory of partial differential equations and probability theory.
In 1958,  Kipriyanov's first work on
fractional order derivatives appeared, in which the concept
of fractional order partial derivatives for functions given in a cube was introduced, starting from the fractional
integral in the sense of Marchaud  and derivatives of the  integer order in the sense of  Sobolev S.L., here we should note that the implemented approach was unique and distinctive in comparison with the regular fractional calculus theory of the time.    The definition of the Kipriyanov fractional derivative  is based upon some
integral identity that relates to the Marchaud fractional derivative. The corresponding integral representation is given, the definition of two functional spaces is given, and
embedding theorems and space completeness theorems are proved for them.
One of these theorems in the one-dimensional case significantly complements
the well-known theorem of Hardy and Littlewood on fractional integrals. In the next
cycle of works from 1959 to 1961,  Kupriyanov I.A. studied
the fractional derivatives in the direction introduced by him  \cite{firstab_lit:kipriyanov1960},\cite{firstab_lit:1kipriyanov1960},\cite{firstab_lit:2kipriyanov1960}, \cite{firstab_lit:2.2kipriyanov1962}.
In 1960   Kupriyanov I.A.  published a paper  on the operator
of  fractional differentiation   \cite{firstab_lit:1kipriyanov1960}, which is a fractional order operator with respect to a second-order elliptic operator with smooth coefficients. It is remarkable that this operator    allows us to study  boundary value problems for differential
equations containing, in addition to partial derivatives,  fractional derivatives. Here, we should interrupt  the description of Kipriyanov's other  achievements and focus on the specific questions regarding his most significant scientific contribution - fractional calculus as an  independently constructed axiomatic theory.\\

\noindent{\bf 2.  Branches in generalizations of the  Riemann-Liouville an Marchaud   operators }\\

The central point of fractional calculus  is a concept of fractional differentiation. In this regard, we should admit that  the  Riemann-Liouville operator of fractional differentiation is  at the origin of the  concept  and  plays a special role in the science. Such operators as Caputo and Marchaud certainly are worth mentioned in the context, the first one is not interesting for us since it is more like a reduction of the Riemann-Liouville operator on smooth functions disappearing at the initial point (if we consider the matter from the point of view that is of functional analysis), but the second one does completely reflect a true mathematical nature of  fractional derivative as a notion, since it has a representation in terms of infinitesimal generator of the corresponding semigroup \cite{kukushkin2021a}. It is clear that considering such an approach we are forced to deal with more general notions of the operator theory and in this way the understanding of the notion of fractional derivative as a fractional power of infinitesimal generator is harmoniously completed, on the one hand.

On the other hand, for a harmony of the narrative,  we should referee an extract of the paper  \cite{Sonin Abel} appealing to  another generalization,  if we interpret the fractional differential  Riemann-Liouville operator as a particular case of the derivative of the convolution operator for which the so called Sonin condition holds \cite{Sonin}.
  We should note  that the second direction in understanding the matter   was    developed  by mathematicians  such as
    Rubin B.S. \cite{firstab_lit: Rubin},\cite{firstab_lit: Rubin 1},\cite{firstab_lit: Rubin 2}, Vakulov B.G. \cite{firstab_lit: Vaculov},   Samko S.G.
     \cite{firstab_lit: Samko M. Murdaev},\cite{firstab_lit: Samko Vakulov B. G.}, Karapetyants N.K.
\cite{firstab_lit: Karapetyants N. K. Rubin B. S. 1},\cite{firstab_lit: Karapetyants N. K. Rubin B. S. 2}.
      Let us remind that the so called mapping theorem  for the Riemann-Liouville operator (the particular case of the Sonin operator)  were firstly studied   by H. Hardy and Littlewood  \cite{firstab_lit H-L1}  and nowadays is known as the Hardy-Littlewood theorem with limit index. However there was an attempt to extend this theorem on some class of  weighted Lebesgue spaces defined as functional spaces endowed with the following norm
  $
  \|f\|_{L_{p}(I\!,\,\beta,\gamma)}:=\|f\|_{L_{p}(I\!,\mu)} ,\,\mu(x)=\omega^{\,\beta,\gamma}(x):=(x-a)^{\beta}(b-x)^{\gamma},\,\beta,\gamma\in \mathbb{R},\,I:=(a,b).
  $
  In this   direction  the mathematicians such as  Rubin B.S., Karapetyants N.K. \cite{firstab_lit: Karapetyants N. K. Rubin B. S. 1}
    had  success, the following problem was considered
    $
    I^{\alpha}_{a+}:L_{p}(I\!,\,\beta,\gamma)\longrightarrow\,  ?
    $
    However the converse  theorem was not! All these create the prerequisite to invent another approach for studying mapping properties of the Riemann-Liouville operator  or, more generally,  integral operators. Thus, trying to solve (at least in particular)  more general problem, in the paper  \cite{kukushkin2020Mz}   we deal  with mapping theorems for  operators acting on Banach spaces in order to obtain afterwards the desired results applicable to integral operators. In this regard the following papers are worth noticing \cite{kukushkin2019axi}, \cite{kukushkin2020axi}, where in additional,  a special technique based on the properties of the Jacobi polynomials was introduced.  Based on this approach,
in the paper \cite{Sonin Abel} we  offer a method of studying the Sonin operator \cite{Sonin}, which is defined as a convolution operator $ _{s}I^{\varrho}_{a+}\varphi := \varrho \ast_{a}\!  \varphi$  under  some conditions (the so called Sonin conditions) imposed on the   kernel  $ \varrho,$ i.e.   there exists a function $\vartheta$ such  that $  \varrho  \ast \vartheta=1.$     The particular case of the Sonin kernel is a kernel of the fractional integral  Riemman-Liouville operator, many other examples can be found in papers \cite{Samko 2003}, \cite{Rubin1982}, the first one gives us a survey considering various types of kernels such as the Bessel-type function, the power-exponential function, the incomplete gamma function e.t.c., the main concept of the second one is to construct a widest class of functions being a Sonin kernel.   Here, we can partly close the matter at this point  having noted that there was a successful attempt to establish a criterion  of the solvability of the Sonin-Abel Equation in the Weighted
Lebesgue Space \cite{Sonin Abel}.

However, let us be back to the first understanding of the matter which is closely  connected with the notion of the Kipriyanov operator.\\

\noindent{\bf 3.  The semigroup approach and the spectral theory }\\

 The  idea  discussed in this paragraph relates      to   a   model that gives us a representation  of a  composition of  fractional differential operators   in terms of the semigroup theory.    For instance we can    represent a second order differential operator as a some kind of  a  transform of   the infinitesimal generator of a shift semigroup. Continuing this line of reasonings we generalized  a   differential operator with a fractional integro-differential composition  in final terms   to some transform of the corresponding  infinitesimal generator   and introduced  a class of  transforms of   m-accretive operators. Further,   we  used   methods obtained in the papers
\cite{firstab_lit(arXiv non-self)kukushkin2018},\cite{kukushkin2019} to  study spectral properties  of non-selfadjoint operators acting  in  a complex  separable Hilbert space, these methods alow us to   obtain an  asymptotic equivalence between   the
real component of the resolvent and the resolvent of the   real component of an operator. Due  to such an approach we  obtain relevant  results since  an asymptotic formula  for   the operator  real component  can be  established in many cases
(see \cite{firstab_lit:2Agranovich2011}, \cite{firstab_lit:Rosenblum}). Thus,   a classification  in accordance with  resolvent  belonging     to  the  Schatten-von Neumann  class was  obtained,   a sufficient condition of completeness of the root vectors system was    formulated.

 The latter approach allows    to construct an abstract  model of a   differential operator with a fractional Kipriyanov integro-differential operator  composition  in  final terms, where  modeling is understood as an  interpretation of   concrete differential  operators in terms of the infinitesimal generator of a corresponding semigroup.  Moreover, we can  consider an  approach in contracting the space of fractionally-differentiable functions which originates  from the analog  created  by  Kipriyanov and goes further up to the semigroup theory generalizations.
   In this paper we  deal with a more general operator --- a  differential operator    with a fractional  integro-differential operator   composition  in   final terms, which covers the corresponding one-dimensional operator. Various types of fractional integro-differential operator compositions were studied by such mathematicians as
 Prabhakar T.R. \cite{firstab_lit:1Prabhakar}, Love E.R. \cite{firstab_lit:5Love}, Erdelyi A. \cite{firstab_lit:15Erdelyi}, McBride A. \cite{firstab_lit:9McBride},
  Dimovski I.H., Kiryakova V.S. \cite{firstab_lit:2Dim-Kir}, Nakhushev A.M. \cite{firstab_lit:nakh2003}.
In particular the aim of this paper is to represent a description of the previously obtained results under a specific point of view related with  Kipriyanov's fractional calculus.      \\

\noindent{\bf 3.  Evolution equations with the operator function in the second term }\\

   Having created a direction of the spectral theory of non-selfadjoint operators, we can consider abstract theoretical results as a base for further research studying such mathematical objects as a Cauchy problem for evolution equation of fractional order in the abstract Hilbert space. We consider in the second term an operator function defined on a special operator class covering a generator transform considered in \cite{kukushkin2021a} and discussed in the previous paragraph, where a corresponding semigroup is supposed to be a $C_{0}$ semigroup of contractions. In its own turn the transform reduces to a linear composition of differential operators of real order in various senses such as  the Riemann-Liouville  fractional differential operator, the Kipriyanov operator,  the Riesz potential, the difference operator   \cite{firstab_lit:1kipriyanov1960},\cite{kukushkin2021a},\cite{firstab_lit:samko1987}.  Moreover, in the paper  \cite{firstab_lit:2kukushkin2022} we broadened the class  of differential operators   having considered  the artificially constructed normal operator that     cannot be covered by the Lidskii results \cite{firstab_lit:1Lidskii}. It should be noted that the Kipriyanov operator is very useful in theoretical constructions as well as in applications since it covers Euclidean spaces and can be considered as a term in a perturbation of a differential operator of an arbitrary odd order acting in n-dimensional Euclidean space. This fact is based upon the the brilliant idea of  Kipriyanov to consider directional coordinates in the n-dimensional Euclidean space, the latter approach is independent on dimension what is an enormous advantage for we can consider compositions of operators having various nature.

    The application part of the theory  involving  fractional integro-differential constructions appeals to the results and problems which can be considered as particular cases of the abstract ones,  the following papers a worth noting within the context
   \cite{firstab_lit:Mamchuev2017a},\cite{firstab_lit:Mamchuev2017},\cite{firstab_lit:Pskhu}. At the same time, we should admit that abstract methods can be "clumsy" for  some peculiarities can be considered only  by a unique technique what forms a main contribution of the specialists dealing with concrete differential equations. Here, we should add that  the  relevance of the abstract  problems  can be expressed convexly by virtue of  the application of the fractional integro-differential compositions  with the   Kipriyanov operator  in  physics and engineering sciences.

Apparently,  in the paper \cite{firstab_lit(frac2023)}  we realized  the idea to broaden the class of fractional integro-differential compositions  having  considered  a notion of    operator function   applicably to a Cauchy problem  for an abstract fractional  evolution equation with an operator function in the second term not containing the time variable, where the derivative in the first term is supposed to be of fractional order. Here, we should note that regarding  to functional spaces we have  that an operator function generates  a variety of operators acting in a corresponding space. In this regard, even a power function gives us an interesting result  \cite{firstab_lit:2kukushkin2022}.
 In the context of the existence and uniqueness theorems, a significant  refinement  that is worth   highlighting  is the obtained  formula for the solution represented by a series on the root vectors. In the absence of the norm convergence of the root vector  series, we need to consider a notion of convergence in   weaker Bari, Riesz, Abel-Lidskii  senses  \cite{firstab_lit:1Lidskii},\cite{firstab_lit:2Agranovich1994},\cite{firstab_lit:1Gohberg1965}.

In spite of the claimed rather applied objectives from the operator theory point of view, we admit that  the  problem of the root vectors expansion for a non-selfadjoint unbounded  operator still remains relevant in the context of the paper.
  It is remarkable that the problem origins nearly  from  the first  half of the last century \cite{firstab_lit:1Katsnelson}, \cite{firstab_lit:2Markus}, \cite{firstab_lit:1Lidskii},      \cite{firstab_lit:1Krein},\cite{firstab_lit:Markus Matsaev},\cite{firstab_lit:2Agranovich1994},\cite{firstab_lit:Shkalikov A.}, \cite{firstab_lit:Motovilov},\cite{firstab_lit(arXiv non-self)kukushkin2018}, \cite{kukushkin2019} \cite{kukushkin2021a}, \cite{firstab_lit:1kukushkin2021}. However, we have a particular interest  when an operator is represented by a linear combination of operators where   a  so-called senior term  is non-selfadjoint  for a case corresponding to a selfadjoint operator was thoroughly studied in the  papers \cite{firstab_lit:1Katsnelson},\cite{firstab_lit:1Krein},\cite{firstab_lit:Markus Matsaev},\cite{firstab_lit:2Markus},\cite{firstab_lit:Motovilov},
\cite{firstab_lit:Shkalikov A.}. In this regard the linear combination of the second order  differential operator and the Kipriyanov operator represents a relevant model  class for which the obtained  spectral theory results \cite{kukushkin2019} created a prerequisite for       further abstract generalizations
 \cite{firstab_lit(arXiv non-self)kukushkin2018},\cite{kukushkin2021a}.

\section{Abstract method}

\vspace{0.5cm}
\noindent{\bf 1. Preliminaries}\\

Let    $ C,C_{i} ,\;i\in \mathbb{N}_{0}$ be     positive constants. We   assume   that  a  value of $C$    can be different in   various formulas and parts of formulas  but   values of $C_{i} $ are  certain. Denote by $ \mathrm{Fr}\,M$   the set of boundary points of the set $M.$    Everywhere further, if the contrary is not stated, we consider   linear    densely defined operators acting on a separable complex  Hilbert space $\mathfrak{H}$. Denote by $ \mathcal{B} (\mathfrak{H})$    the set of linear bounded operators   on    $\mathfrak{H}.$  Denote by
    $\tilde{L}$   the  closure of an  operator $L.$ We establish the following agreement on using  symbols $\tilde{L}^{i}:= (\tilde{L})^{i},$ where $i$ is an arbitrary symbol.  Denote by    $    \mathrm{D}   (L),\,   \mathrm{R}   (L),\,\mathrm{N}(L)$      the  {\it domain of definition}, the {\it range},  and the {\it kernel} or {\it null space}  of an  operator $L,$ respectively. The deficiency (codimension) of $\mathrm{R}(L),$ dimension of $\mathrm{N}(L)$ are denoted by $\mathrm{def}\, L,\;\mathrm{nul}\,L$ respectively. Assume that $L$ is a closed   operator acting on $\mathfrak{H},\,\mathrm{N}(L)=0,$  let us define a Hilbert space
$
 \mathfrak{H}_{L}:= \big \{f,g\in \mathrm{D}(L),\,(f,g)_{ \mathfrak{H}_{L}}=(Lf,Lg)_{\mathfrak{H} } \big\}.
$
Consider a pair of complex Hilbert spaces $\mathfrak{H},\mathfrak{H}_{+},$ the notation
$
\mathfrak{H}_{+}\subset\subset\mathfrak{ H}
$
   means that $\mathfrak{H}_{+}$ is dense in $\mathfrak{H}$ as a set of    elements and we have a bounded embedding provided by the inequality
$
\|f\|_{\mathfrak{H}}\leq C_{0}\|f\|_{\mathfrak{H}_{+}},\,C_{0}>0,\;f\in \mathfrak{H}_{+},
$
moreover   any  bounded  set with respect to the norm $\mathfrak{H}_{+}$ is compact with respect to the norm $\mathfrak{H}.$
  Let $L$ be a closed operator, for any closable operator $S$ such that
$\tilde{S} = L,$ its domain $\mathrm{D} (S)$ will be called a core of $L.$ Denote by $\mathrm{D}_{0}(L)$ a core of a closeable operator $L.$ Let    $\mathrm{P}(L)$ be  the resolvent set of an operator $L$ and
     $ R_{L}(\zeta),\,\zeta\in \mathrm{P}(L),\,[R_{L} :=R_{L}(0)]$ denotes      the resolvent of an  operator $L.$ Denote by $\lambda_{i}(L),\,i\in \mathbb{N} $ the eigenvalues of an operator $L.$
 Suppose $L$ is  a compact operator and  $N:=(L^{\ast}L)^{1/2},\,r(N):={\rm dim}\,  \mathrm{R}  (N);$ then   the eigenvalues of the operator $N$ are called   the {\it singular  numbers} ({\it s-numbers}) of the operator $L$ and are denoted by $s_{i}(L),\,i=1,\,2,...\,,r(N).$ If $r(N)<\infty,$ then we put by definition     $s_{i}=0,\,i=r(N)+1,2,...\,.$
 Let  $\nu(L)$ denotes   the sum of all algebraic multiplicities of an  operator $L.$ Denote by $n(r)$ a function equals to the quantity of the elements of the sequence $\{a_{n}\}_{1}^{\infty},\,|a_{n}|\uparrow\infty$ within the circle $|z|<r.$ Let $A$ be a compact operator, denote by $n_{A}(r),$   {\it counting function}   a function $n(r)$ corresponding to the sequence  $\{s^{-1}_{i}(A)\}_{1}^{\infty}.$
  Let  $\mathfrak{S}_{p}(\mathfrak{H}),\, 0< p<\infty $ be       a Schatten-von Neumann    class and      $\mathfrak{S}_{\infty}(\mathfrak{H})$ be the set of compact operators.
   Denote by $\tilde{\mathfrak{S}}_{\rho}(\mathfrak{H})$ the class of the operators such that
$
 A\in \tilde{\mathfrak{S}}_{\rho}(\mathfrak{H}) \Rightarrow\{A\in\mathfrak{S}_{\rho+\varepsilon},\,A \overline{\in} \,\mathfrak{S}_{\rho-\varepsilon},\,\forall\varepsilon>0 \}.
$
In accordance with \cite{firstab_lit:1kukushkin2021}, we will call it   {\it Schatten-von Neumann    class of the convergence exponent}.
Suppose  $L$ is  an   operator with a compact resolvent and
$s_{n}(R_{L})\leq   C \,n^{-\mu},\,n\in \mathbb{N},\,0\leq\mu< \infty;$ then
 we
 denote by  $\mu(L) $   order of the     operator $L$  (see \cite{firstab_lit:Shkalikov A.}).
 Denote by  $ \mathfrak{Re} L  := \left(L+L^{*}\right)/2,\, \mathfrak{Im} L  := \left(L-L^{*}\right)/2 i$
  the  real  and   imaginary Hermitian  components    of an  operator $L$  respectively.
In accordance with  the terminology of the monograph  \cite{firstab_lit:kato1980}, the set $\Theta(L):=\{z\in \mathbb{C}: z=(Lf,f)_{\mathfrak{H}},\,f\in  \mathrm{D} (L),\,\|f\|_{\mathfrak{H}}=1\}$ is called the  {\it numerical range}  of an   operator $L.$
  An  operator $L$ is called    {\it sectorial}    if its  numerical range   belongs to a  closed
sector     $\mathfrak{ L}_{\iota}(\theta):=\{\zeta:\,|\arg(\zeta-\iota)|\leq\theta<\pi/2\} ,$ where      $\iota$ is the vertex   and  $ \theta$ is the semi-angle of the sector   $\mathfrak{ L}_{\iota}(\theta).$ If we want to stress the  correspondence  between $\iota$ and $\theta,$  then   we will write $\theta_{\iota}.$
 An operator $L$ is called  {\it bounded from below}   if the following relation  holds  $\mathrm{Re}(Lf,f)_{\mathfrak{H}}\geq \gamma_{L}\|f\|^{2}_{\mathfrak{H}},\,f\in  \mathrm{D} (L),\,\gamma_{L}\in \mathbb{R},$  where $\gamma_{L}$ is called a lower bound of $L.$ An operator $L$ is called  {\it   accretive}   if  $\gamma_{L}=0.$
 An operator $L$ is called  {\it strictly  accretive}   if  $\gamma_{L}>0.$      An  operator $L$ is called    {\it m-accretive}     if the next relation  holds $(A+\zeta)^{-1}\in \mathcal{B}(\mathfrak{H}),\,\|(A+\zeta)^{-1}\| \leq   (\mathrm{Re}\zeta)^{-1},\,\mathrm{Re}\zeta>0. $
    An operator $L$ is called     {\it symmetric}     if one is densely defined and the following  equality  holds $(Lf,g)_{\mathfrak{H}}=(f,Lg)_{\mathfrak{H}},\,f,g\in   \mathrm{D}  (L).$  Consider a   sesquilinear form   $ s [\cdot,\cdot]$ (see \cite{firstab_lit:kato1980} )
defined on a linear manifold  of the Hilbert space $\mathfrak{H}.$
Let   $  \mathfrak{h}=( s + s ^{\ast})/2,\, \mathfrak{k}   =( s - s ^{\ast})/2i$
   be a   real  and    imaginary component     of the   form $  s $ respectively, where $ s^{\ast}[u,v]=s \overline{[v,u]},\;\mathrm{D}(s ^{\ast})=\mathrm{D}(s).$ Denote by $   s  [\cdot ]$ the  quadratic form corresponding to the sesquilinear form $ s  [\cdot,\cdot].$ According to these definitions, we have $
 \mathfrak{h}[\cdot]=\mathrm{Re}\,s[\cdot],\,  \mathfrak{k}[\cdot]=\mathrm{Im}\,s[\cdot].$ Denote by $\tilde{s}$ the  closure   of a   form $s.$  The range of a quadratic form
  $ s [f],\,f\in \mathrm{D}(s),\,\|f\|_{\mathfrak{H}}=1$ is called   the {\it range} of the sesquilinear form  $s $ and is denoted by $\Theta(s).$
 A  form $s$ is called    {\it sectorial}    if  its    range  belongs to   a sector  having  a vertex $\iota$  situated at the real axis and a semi-angle $0\leq\theta<\pi/2.$   Due to Theorem 2.7 \cite[p.323]{firstab_lit:kato1980}  there exist unique    m-sectorial operators  $T_{s},T_{ \mathfrak{h}} $  associated  with   the  closed sectorial   forms $s,  \mathfrak{h}$   respectively.   The operator  $T_{  \mathfrak{h}} $ is called  a {\it real part} of the operator $T_{s}$ and is denoted by  $\mathrm{Re}\, T_{s}.$

Assume that  $T_{t},\,(0\leq t<\infty)$ is a semigroup of bounded linear operators on    $\mathfrak{H},$    by definition put
$$
Af=-\lim\limits_{t\rightarrow+0} \left(\frac{T_{t}-I}{t}\right)f,
$$
where $\mathrm{D}(A)$  is a set of elements for which  the last limit  exists in the sense of the norm  $\mathfrak{H}.$    In accordance with definition \cite[p.1]{Pasy} the operator $-A$ is called the  {\it infinitesimal  generator} of the semigroup $T_{t}.$

  Let $f_{t} :I\rightarrow \mathfrak{H},\,t\in I:=[a,b],\,-\infty< a <b<\infty.$ The following integral is understood in the Riemann  sense as a limit of partial sums
\begin{equation}\label{11.001}
\sum\limits_{i=0}^{n}f_{\xi_{i}}\Delta t_{i}  \stackrel{\mathfrak{H}}{ \longrightarrow}  \int\limits_{I}f_{t}dt,\,\lambda\rightarrow 0,
\end{equation}
where $(a=t_{0}<t_{1}<...<t_{n}=b)$ is an arbitrary splitting of the segment $I,\;\lambda:=\max\limits_{i}(t_{i+1}-t_{i}),\;\xi_{i}$ is an arbitrary point belonging to $[t_{i},t_{i+1}].$
The sufficient condition of the last integral existence is a continuous property (see\cite[p.248]{firstab_lit:Krasnoselskii M.A.}) i.e.
$
f_{t}\stackrel{\mathfrak{H}}{ \longrightarrow}f_{t_{0}},\,t\rightarrow t_{0},\;\forall t_{0}\in I.
$
The improper integral is understood as a limit
\begin{equation}\label{11.031}
 \int\limits_{a}^{b}f_{t}dt\stackrel{\mathfrak{H}}{ \longrightarrow} \int\limits_{a}^{c}f_{t}dt,\,b\rightarrow c,\,c\in [-\infty,\infty].
\end{equation}

Using     notations of the paper     \cite{firstab_lit:kipriyanov1960}, we assume that $\Omega$ is a  convex domain of the  $n$ -  dimensional Euclidean space $\mathbb{E}^{n}$, $P$ is a fixed point of the boundary $\partial\Omega,$
$Q(r,\mathbf{e})$ is an arbitrary point of $\Omega.$  Let  $\mathrm{d}:=\mathrm{diam}\Omega,$ we denote by $\mathbf{e}$   a unit vector having a direction from  $P$ to $Q,$ denote by $r=|P-Q|$   the Euclidean distance between the points $P,Q,$ and   use the shorthand notation    $T:=P+\mathbf{e}t,\,t\in \mathbb{R}.$
We   consider the Lebesgue  classes   $L_{p}(\Omega),\;1\leq p<\infty $ of  complex valued functions.  For the function $f\in L_{p}(\Omega),$    we have
\begin{equation}\label{1a}
\int\limits_{\Omega}|f(Q)|^{p}dQ=\int\limits_{\omega}d\chi\int\limits_{0}^{d(\mathbf{e})}|f(Q)|^{p}r^{n-1}dr<\infty,
\end{equation}
where $d\chi$   is an element of   solid angle of
the unit sphere  surface (the unit sphere belongs to $\mathbb{E}^{n}$)  and $\omega$ is a  surface of this sphere,   $d:=d(\mathbf{e})$  is the  length of the  segment of the  ray going from the point $P$ in the direction
$\mathbf{e}$ within the domain $\Omega.$
Without loss of  generality, we consider only those directions of $\mathbf{e}$ for which the inner integral on the right-hand  side of equality \eqref{1a} exists and is finite. It is  the well-known fact that  these are almost all directions.
 We use a shorthand  notation  $P\cdot Q=P^{i}Q_{i}=\sum^{n}_{i=1}P_{i}Q_{i}$ for the inner product of the points $P=(P_{1},P_{2},...,P_{n}),\,Q=(Q_{1},Q_{2},...,Q_{n})$ which     belong to  $\mathbb{E}^{n}.$
     Denote by  $D_{i}f$  a weak partial derivative of the function $f$ with respect to a coordinate variable with index   $1\leq i\leq n.$
We  assume that all functions have  a zero extension outside  of $\bar{\Omega}.$
Everywhere further,   unless  otherwise  stated,  we   use  notations of the papers   \cite{firstab_lit:1Gohberg1965},  \cite{firstab_lit:kato1980},  \cite{firstab_lit:kipriyanov1960}, \cite{firstab_lit:1kipriyanov1960},
\cite{firstab_lit:samko1987}.\\

Bellow, we represent  the  conditions of Theorem 1 \cite{kukushkin2021a} that  gives us a description of spectral properties, in terms of the real part order, of a non-selfadjoint operator $L$ acting in $\mathfrak{H}.$   \\

 \noindent  $ (\mathrm{H}1) $ There  exists a Hilbert space $\mathfrak{H}_{+}\subset\subset\mathfrak{ H}$ and a linear manifold $\mathfrak{M}$ that is  dense in  $\mathfrak{H}_{+}.$ The operator $L$ is defined on $\mathfrak{M}.$    \\

 \noindent  $( \mathrm{H2} )  \,\left|(Lf,g)_{\mathfrak{H}}\right|\! \leq \! C_{1}\|f\|_{\mathfrak{H}_{+}}\|g\|_{\mathfrak{H}_{+}},\,
      \, \mathrm{Re}(Lf,f)_{\mathfrak{H}}\!\geq\! C_{2}\|f\|^{2}_{\mathfrak{H}_{+}} ,\,f,g\in  \mathfrak{M},\; C_{1},C_{2}>0.
$
\\

Here, we should remark that since there is no general statement claiming  that the intersection of the domain of definitions of an  operator and its adjoint is a dense set, then we cannot restrict the reasonings considering Hermitian real  component but  compelled to involve the notion of the  operator real part. This is why it is rather reasonable to suggest     the the issue should  be   undergone    to a comprehensive analysis.

Consider  a condition  $\mathfrak{M}\subset \mathrm{D}( W ^{\ast}),$ in this case the real Hermitian component  $\mathcal{H}:=\mathfrak{Re }\,W$ of the operator is defined on $\mathfrak{M},$ the fact is that $\tilde{\mathcal{H}}$ is selfadjoint,    bounded  from bellow (see Lemma  3 \cite{firstab_lit(arXiv non-self)kukushkin2018}), where $H=Re W.$  Hence a corresponding sesquilinear  form (denote this form by $h$) is symmetric and  bounded from bellow also (see Theorem 2.6 \cite[p.323]{firstab_lit:kato1980}). It can be easily shown  that $h\subset   \mathfrak{h},$  but using this fact    we cannot claim in general that $\tilde{\mathcal{H}}\subset H$ (see \cite[p.330]{firstab_lit:kato1980} ). We just have an inclusion   $\tilde{\mathcal{H}}^{1/2}\subset H^{1/2}$     (see \cite[p.332]{firstab_lit:kato1980}). Note that the fact $\tilde{\mathcal{H}}\subset H$ follows from a condition $ \mathrm{D}_{0}(\mathfrak{h})\subset \mathrm{D}(h) $ (see Corollary 2.4 \cite[p.323]{firstab_lit:kato1980}).
 However, it is proved (see proof of Theorem  4 \cite{firstab_lit(arXiv non-self)kukushkin2018}) that relation H2 guaranties that $\tilde{\mathcal{H}}=H.$ Note that the last relation is very useful in applications, since in most concrete cases we can find a concrete form of the operator $\mathcal{H}.$

\vspace{0.5cm}

\noindent{\bf 2. Intrinsic properties of the Kipriyanov operator}\\

Here, we study a case $\alpha\in (0,1).$ Assume that  $\Omega\subset \mathbb{E}^{n}$ is  a convex domain, with a sufficient smooth boundary ($ C ^{3}$ class)   of the n-dimensional Euclidian space. For the sake of the simplicity we consider that $\Omega$ is bounded, but  the results  can be extended     to some type of    unbounded domains.
In accordance with the definition given in  the paper  \cite{firstab_lit:1kukushkin2018}, we consider the directional  fractional integrals.  By definition, put
$$
 (\mathfrak{I}^{\alpha}_{0+}f)(Q  ):=\frac{1}{\Gamma(\alpha)} \int\limits^{r}_{0}\frac{f (P+t \mathbf{e} )}{( r-t)^{1-\alpha}}\left(\frac{t}{r}\right)^{n-1}\!\!\!\!dt,\,(\mathfrak{I}^{\alpha}_{d-}f)(Q  ):=\frac{1}{\Gamma(\alpha)} \int\limits_{r}^{d }\frac{f (P+t\mathbf{e})}{(t-r)^{1-\alpha}}\,dt,
$$
$$
\;f\in L_{p}(\Omega),\;1\leq p\leq\infty.
$$
The properties of these operators  are  described  in detail in the papers  \cite{firstab_lit:1kukushkin2018},\cite{Sam2017}. Similarly to the monograph \cite{firstab_lit:samko1987} we consider   left-side  and  right-side cases. For instance, $\mathfrak{I}^{\alpha}_{0+}$ is called  a left-side directional  fractional integral. We suppose  $\mathfrak{I}^{0}_{0+} =I.$ Nevertheless,   this    fact can be easily proved dy virtue of  the reasonings  corresponding to the one-dimensional case and   given in \cite{firstab_lit:samko1987}. We also consider integral operators with a weighted factor (see \cite[p.175]{firstab_lit:samko1987}) defined by the following formal construction
$$
 \left(\mathfrak{I}^{\alpha}_{0+}\mu f\right) (Q  ):=\frac{1}{\Gamma(\alpha)} \int\limits^{r}_{0}
 \frac{(\mu f) (P+t\mathbf{e})}{( r-t)^{1-\alpha}}\left(\frac{t}{r}\right)^{n-1}\!\!\!\!dt,
$$
where $\mu$ is a real-valued  function.
We introduce   the classes of functions representable by the directional fractional integrals.
 \begin{equation*}\label{5z}
  \mathfrak{I}^{\alpha}_{0+}(L_{p}  ):=\left\{ u:\,u(Q)=(\mathfrak{I}^{\alpha}_{0+}g)(Q  ) \right\},\;
 \mathfrak{I }  ^{\alpha}_{ d   -} (L_{p}  ) =\left\{ u:\,u(Q)=(\mathfrak{I}^{\alpha}_{d-}g)(Q  )  \right\},\, g\in L_{p}(\Omega),\,1\leq p\leq\infty.
 \end{equation*}

Define the following auxiliary  operators acting in  $L_{p}(\Omega)$  and   depended on the parameter $\varepsilon>0.$  In the left-side case
 \begin{equation}\label{7z}
(\psi^{+}_{  \varepsilon }f)(Q)=  \left\{ \begin{aligned}
 \int\limits_{0}^{r-\varepsilon }\frac{ f (Q)r^{n-1}- f(T)t^{n-1}}{(  r-t)^{\alpha +1}r^{n-1}}  dt,\;\varepsilon\leq r\leq d  ,\\
   \frac{ f(Q)}{\alpha} \left(\frac{1}{\varepsilon^{\alpha}}-\frac{1}{ r ^{\alpha} }    \right),\;\;\;\;\;\;\;\;\;\;\;\;\;\;\; 0\leq r <\varepsilon .\\
\end{aligned}
 \right.
\end{equation}
In the right-side case
\begin{equation*}
 (\psi^{-}_{  \varepsilon }f)(Q)=  \left\{ \begin{aligned}
 \int\limits_{r+\varepsilon }^{d }\frac{ f (Q)- f(T)}{( t-r)^{\alpha +1}} dt,\;0\leq r\leq d -\varepsilon,\\
   \frac{ f(Q)}{\alpha} \left(\frac{1}{\varepsilon^{\alpha}}-\frac{1}{(d -r)^{\alpha} }    \right),\;\;\;d -\varepsilon <r \leq d .\\
\end{aligned}
 \right.
 \end{equation*}

Using the definitions of the monograph  \cite[p.181]{firstab_lit:samko1987}  we consider the following operators.  In the left-side case
 \begin{equation}\label{8}
 ( \mathfrak{D} ^{\alpha}_{0+\!,\,\varepsilon}f)(Q)=\frac{1}{\Gamma(1-\alpha)}f(Q) r ^{-\alpha}+\frac{\alpha}{\Gamma(1-\alpha)}(\psi^{+}_{  \varepsilon }f)(Q).
 \end{equation}
 In the right-side case
 \begin{equation*}
 ( \mathfrak{D }^{\alpha}_{d-\!,\,\varepsilon}f)(Q)=\frac{1}{\Gamma(1-\alpha)}f(Q)(d-r)^{-\alpha}+\frac{\alpha}{\Gamma(1-\alpha)}(\psi^{-}_{  \varepsilon }f)(Q).
 \end{equation*}
 The left-side  and  right-side fractional derivatives  are  understood  respectively  as the  following  limits
 \begin{equation}\label{8.1}
 \mathfrak{D }^{\alpha}_{0+}f=\lim\limits_{\stackrel{\varepsilon\rightarrow 0}{ (L_{p}) }} \mathfrak{D }^{\alpha}_{0+\!,\,\varepsilon} f  ,\; \mathfrak{D }^{\alpha}_{d-}f=\lim\limits_{\stackrel{\varepsilon\rightarrow 0}{ (L_{p}) }} \mathfrak{D }^{\alpha}_{d-\!,\,\varepsilon} f,\,1\leq p<\infty.
\end{equation}

Consider the  Kipriyanov  fractional differential operator     defined in  the paper \cite{firstab_lit:1kipriyanov1960}  by  the formal expression
\begin{equation*}
\mathfrak{D}^{\alpha}(Q)=\frac{\alpha}{\Gamma(1-\alpha)}\int\limits_{0}^{r} \frac{[f(Q)-f(T)]}{(r - t)^{\alpha+1}} \left(\frac{t}{r} \right) ^{n-1} dt+
C^{(\alpha)}_{n} f(Q) r ^{ -\alpha}\!,\, P\in\partial\Omega,
\end{equation*}
where
$
C^{(\alpha)}_{n} = (n-1)!/\Gamma(n-\alpha).
$
It is remarkable that   Theorem 2   \cite{firstab_lit:1kipriyanov1960} establishes the mapping properties of the Kipriyanov operator, here we represent its statement in the explicit form: under the assumptions
\begin{equation}\label{2z}
 lp\leq n,\;0<\alpha<l- \frac{n}{p} +\frac{n}{q}, \,q>p,
\end{equation}
we have that
     for sufficiently small $\delta>0$ the following inequality holds
\begin{equation}\label{3z}
\|\mathfrak{D}^{\alpha}f\|_{L_{q}(\Omega)}\leq \frac{K}{\delta^{\nu}}\|f\|_{L_{p}(\Omega)}+\delta^{1-\nu}\|f\|_{L^{l}_{p}(\Omega)},\, f\in \dot{W}_{p}^{\,l}  (\Omega),
\end{equation}
where
\begin{equation}\label{4z}
 \nu=\frac{n}{l}\left(\frac{1}{p}-\frac{1}{q} \right)+\frac{\alpha+\beta}{l}.
\end{equation}
The constant  $K$ does not depend on $\delta,\,f;$   the point $P\in\partial\Omega ;\;\beta$ is an arbitrarily small fixed positive number. It is remarkable that Lemma 2.5  \cite{firstab_lit:1kukushkin2018} establishes the connection between the fractional differential operators, more precisely  it  establishes the following relation
$$
 (\mathfrak{D}^{ \alpha }f)(Q)= \left(\mathfrak{D}^{ \alpha }_{0+} f\right)(Q), \,f\in   \dot{W}_{p}^{l} (\Omega),
$$
what leads us  to the inclusion  $\mathfrak{D}^{ \alpha }\subset\mathfrak{D}^{ \alpha }_{0+}.$

The  following theorem   \cite{firstab_lit:1kukushkin2018}  establishes the mapping properties of   directional fractional integral operators.
 \begin{teo}\label{T11}  The following estimates hold
  \begin{equation*}\label{9z}
 \| \mathfrak{I}^{\alpha}_{0 +}u\|_{L_{p}(\Omega)}\leq C_{\alpha,\mathrm{d}}\|u \|_{L_{p}(\Omega)},\;\|   \mathfrak{I} ^{\alpha}_{d -}u\|_{L_{p}(\Omega)}\leq C_{\alpha,\mathrm{d}}\|u \|_{L_{p}(\Omega)},\;C_{\alpha,\mathrm{d}}=   \mathrm{d}  ^{\alpha}/ \Gamma(\alpha+1),\,1\leq p<\infty  .
 \end{equation*}
 \end{teo}
The proof of the  following   so-called   representation theorem given in  \cite{firstab_lit:1kukushkin2018}  implements the scheme of the proof corresponding to the one-dimensional case  invented by Rubin B.S. \cite{firstab_lit: Rubin72}, \cite{Rubin73}. The author's own  merit is a creation of  the adopted version   applicable to the Kipriyanov operator, we represent it in the expanded form since it may be treated as the intersection of the classical fractional calculus with the theory invented  by Kipriyanov I.A.

\begin{teo}\label{T22}
Suppose $f\in L_{p}(\Omega),$  there  exists   $\lim\limits_{\varepsilon\rightarrow  0}\psi^{+}_{  \varepsilon }f$ or $\lim\limits_{\varepsilon\rightarrow  0}\psi^{-}_{  \varepsilon }f$ with respect to the norm $L_{p}(\Omega),\, 1\leq p<\infty,$ then   $f\in \mathfrak{I} ^{\alpha}_{0 +}(L_{p}) $ or $f\in \mathfrak{I }^{\alpha}_{d -}(L_{p})$ respectively.
\end{teo}
\begin{proof}
 Let $f\in L_{p}(\Omega)$ and $\lim\limits_{\stackrel{\varepsilon\rightarrow 0}{ (L_{p}) }}\psi^{+}_{  \varepsilon }f=\psi.$ Consider the function
$$
(\varphi^{+}_{ \varepsilon}f)(Q)=\frac{1}{\Gamma(1-\alpha)}\left\{\frac{f(Q)}{ r ^{\alpha}}+\alpha (\psi^{+}_{ \varepsilon }f)(Q) \right\}.
 $$
 Taking into account     \eqref{7z}, we can easily prove  that $\varphi^{+}_{  \varepsilon }f\in L_{p}(\Omega).$ Obviously,    there  exists the limit
 $\varphi^{+}_{  \varepsilon }f\rightarrow \varphi\in L_{p}(\Omega),\,\varepsilon\rightarrow 0.$ Taking into account Theorem \ref{T11}, we can  complete the proof,  if we  show that
 \begin{equation}\label{10.01}
 \mathfrak{I}^{\alpha}_{0+}\varphi^{+}_{ \varepsilon }f  \stackrel{L_{p}}{\longrightarrow} f,\,\varepsilon\rightarrow0.
 \end{equation}
 In the case  $ \varepsilon\leq r\leq d ,$ we have
$$
(\mathfrak{I}^{\alpha}_{0 +}\varphi^{+}_{ \varepsilon }f)(Q)\cdot\frac{\pi r^{n-1}}{\sin\alpha\pi}=
 \int\limits_{\varepsilon}^{r}\frac{f (P+y\mathbf{e})y ^{n-1-\alpha}}{( r-y)^{1-\alpha} } dy
 +\alpha\int\limits_{\varepsilon}^{r}( r-y)^{\alpha-1}  dy\int\limits_{0 }^{y-\varepsilon }\frac{ f (P+y\mathbf{e})y^{n-1}- f(T)t^{n-1}}{( y-t )^{\alpha +1}} dt+
$$
$$
 +\frac{1}{\varepsilon^{\alpha}}\int\limits_{0}^{\varepsilon  }f (P+y\mathbf{e})( r-y)^{\alpha-1} y^{n-1}  dy =I.
$$
By direct calculation, we obtain
 \begin{equation}\label{10z}
I=   \frac{1}{\varepsilon^{\alpha}}\int\limits_{0}^{r  }f (P+y\mathbf{e})( r-y)^{\alpha-1}y^{n-1}   dy  -
 \alpha\int\limits_{\varepsilon}^{r}( r-y)^{\alpha-1} dy\int\limits_{0 }^{y-\varepsilon }\frac{  f(T)}{(  y-t)^{\alpha +1}}t^{n-1} dt .
\end{equation}
 Changing the variable  in  the second  integral,   we have
$$
 \alpha\int\limits_{\varepsilon}^{r}( r-y)^{\alpha-1} dy\int\limits_{0 }^{y-\varepsilon }\frac{  f(T)}{(  y-t)^{\alpha +1}}t^{n-1} dt
 =\alpha\int\limits_{0}^{r-\varepsilon}( r-y-\varepsilon)^{\alpha-1} dy\int\limits_{0 }^{y  }\frac{  f(T)}{(  y+\varepsilon-t)^{\alpha +1}}t^{n-1} dt=
$$
\begin{equation}\label{11z}
=\alpha\int\limits_{0}^{r-\varepsilon}f(T)t^{n-1} dt\int\limits_{t }^{r-\varepsilon  }\frac{( r-y-\varepsilon)^{\alpha-1}   }{(  y+\varepsilon-t)^{\alpha +1}}dy
=\alpha\int\limits_{0}^{r-\varepsilon}f(T)t^{n-1} dt\int\limits_{t +\varepsilon}^{r  } ( r-y )^{\alpha-1}   (  y -t)^{-\alpha -1} dy .
\end{equation}
Applying formula   (13.18) \cite[p.184]{firstab_lit:samko1987}, we get
\begin{equation}\label{12z}
 \int\limits_{t +\varepsilon}^{r  } ( r-y )^{\alpha-1}   (  y -t)^{-\alpha -1} dy=
  \frac{1}{\alpha \varepsilon^{\alpha}}\cdot\frac{(r-t-\varepsilon)^{\alpha}}{ r-t }.
\end{equation}
 Combining    relations \eqref{10z},\eqref{11z},\eqref{12z}, using the change of the variable     $t=r-\varepsilon\tau, $ we get
$$
(\mathfrak{I}^{\alpha}_{0 +}\varphi^{+}_{ \varepsilon }f)(Q)\cdot\frac{\pi r^{n-1}}{\sin\alpha\pi}
 =\frac{1}{\varepsilon^{\alpha}}\left\{ \int\limits_{0}^{r  }f (P+y\mathbf{e})(r-y)^{\alpha-1}y^{n-1}   dy- \int\limits_{0}^{r-\varepsilon  } \frac{f(T)(r-t-\varepsilon)^{\alpha}}{ r-t } t^{n-1}dt \right\}=
$$
\begin{equation}\label{13z}
=\frac{1}{ \varepsilon^{\alpha}} \int\limits_{0 }^{r  } \frac{f(T)\left[(r -t)^{\alpha}-(r-t-\varepsilon)_{+}^{\alpha}\right]}{ r-t }t^{n-1}dt
 =\int\limits_{0 }^{r/\varepsilon   }\frac{\tau^{\alpha}-(\tau-1)_{+}^{\alpha}}{ \tau } f(P+[r-\varepsilon \tau ]\mathbf{e})(r-\varepsilon \tau)^{n-1} d\tau,
 $$
 $$
 \tau_{+}=\left\{\begin{array}{cc}\tau,\;\tau\geq0;\\[0,25cm] 0,\;\tau<0\,.\end{array}\right.
 \end{equation}
Consider the auxiliary function $\mathcal{K}$ defined in the paper   \cite[p.105]{firstab_lit:samko1987}
 \begin{equation}\label{14z}
 \mathcal{K}(t)= \frac{\sin\alpha\pi}{\pi }\cdot\frac{ t_{+}^{\alpha}-(t-1)_{+}^{\alpha}}{ t },\; \int\limits_{0 }^{\infty  }\mathcal{K}(t)dt=1;\;\mathcal{K}(t)>0.
\end{equation}
  Combining  \eqref{13z},\eqref{14z} and taking into account that    $f$ has the zero extension outside of  $\bar{\Omega},$ we obtain
 \begin{equation}\label{15z}
 (\mathfrak{I}^{\alpha}_{0+}\varphi^{+}_{ \varepsilon }f)(Q)-f(Q) =  \int\limits_{0 }^{\infty  }\mathcal{K}(t) \left\{f(P+[r-\varepsilon t]\mathbf{e})(1-\varepsilon t/r)_{+}^{n-1}-f(P+ r \mathbf{e})  \right\}dt.
\end{equation}
Consider the case  $0\leq  r <\varepsilon.$ Taking into account \eqref{7z}, we get
 \begin{equation}\label{16z}
(\mathfrak{I}^{\alpha}_{0+}\varphi^{+}_{ \varepsilon }f)(Q)-f (Q) =
\frac{\sin\alpha\pi}{\pi\varepsilon^{\alpha}} \int\limits_{0}^{r }\frac{f (T)}{(r-t)^{1-\alpha} } \left(\frac{t}{r} \right)^{n-1} dt-f (Q)=
$$
$$
=\frac{\sin\alpha\pi}{\pi\varepsilon^{\alpha}} \int\limits_{0}^{r }\frac{f (P+[r-t]\mathbf{e})}{t^{1-\alpha} }\left(\frac{r-t}{r} \right)^{n-1} dt-f (Q).
\end{equation}
Consider the domains
 \begin{equation}\label{17z}
 \Omega_{\varepsilon} :=\{Q\in\Omega,\,d(\mathbf{e})\geq\varepsilon \},\;\tilde{\Omega}_{  \varepsilon }=\Omega\setminus \Omega_{\varepsilon}.
 \end{equation}
In accordance with    this definition  we can  divide the surface $\omega$ into two  parts $ \omega_{\varepsilon}$ and  $\tilde{\omega}_{ \varepsilon },$ where $ \omega_{\varepsilon}$ is  the subset of  $\omega$ such that   $d(\mathbf{e})\geq\varepsilon$ and $\tilde{\omega}_{ \varepsilon }$ is  the subset of  $\omega$ such that  $d(\mathbf{e}) <\varepsilon.$
Using  \eqref{15z},\eqref{16z}, we get
\begin{equation}\label{18z}
  \|(\mathfrak{I}^{\alpha}_{0+}\varphi^{+}_{ \varepsilon }f) -f\|^{p}_{L_{p}(\Omega)}
 =   \int\limits_{\omega_{\varepsilon}}d\chi\int\limits_{\varepsilon}^{ d}
  \left|\int\limits_{0 }^{\infty  }\mathcal{K}(t)[f(Q- \varepsilon t \mathbf{e})(1-\varepsilon t/r)_{+}^{n-1}-f(Q)]dt\right|^{p}r^{n-1}dr+
$$
$$
+ \int\limits_{\omega_{\varepsilon}}d\chi\int\limits_{0}^{\varepsilon }
\left|  \frac{\sin\alpha\pi}{\pi\varepsilon^{\alpha}}\int\limits_{0}^{r }
\frac{f (P+[r-t]\mathbf{e})}{t^{1-\alpha} }\left(\frac{r-t}{r} \right)^{n-1} dt-f (Q)\right|^{p}r^{n-1}dr+
$$
$$
+  \int\limits_{\tilde{\omega}_{\varepsilon} }d\chi\int\limits_{0}^{d }
\left| \frac{\sin\alpha\pi}{\pi\varepsilon^{\alpha}}\int\limits_{0}^{r }\frac{f (P+[r-t]\mathbf{e})}{t^{1-\alpha} }\left(\frac{r-t}{r} \right)^{n-1} dt-f (Q) \right|^{p}r^{n-1}dr =I_{1}+I_{2}+I_{3}.
 \end{equation}
Consider $ I_{1},$ using the generalized Minkowski inequality,  we get
$$
  I^{\frac{1 }{p}}_{1} \leq  \int\limits_{0 }^{\infty  }\mathcal{K}(t)
  \left(\int\limits_{\omega_{\varepsilon} }d\chi\int\limits_{\varepsilon}^{ d}
  |f(Q- \varepsilon t \mathbf{e})(1-\varepsilon t/r)_{+}^{n-1}-f(Q)|^{p}r^{n-1}dr \right)^{\frac{1 }{p}} dt.
$$
Let us define the function
$$
 h(\varepsilon,t):= \mathcal{K}(t)\left(\int\limits_{\omega_{\varepsilon} }d\chi\int\limits_{\varepsilon}^{ d}
  |f(Q- \varepsilon t \mathbf{e})(1-\varepsilon t/r)_{+}^{n-1}-f(Q)|^{p}r^{n-1}dr \right)^{\frac{1 }{p}} dt.
$$
It can easily be checked that the following inequalities hold
\begin{equation}\label{19z}
  |h(\varepsilon,t)|\leq 2\mathcal{K}(t) \| f\|_{L_{p}(\Omega)},\;\forall\varepsilon>0;
\end{equation}
  \begin{equation*}
  |h(\varepsilon,t)|\leq  \left(\int\limits_{\omega_{\varepsilon} }d\chi\int\limits_{\varepsilon}^{ d}
 \left |(1-\varepsilon t/r)_{+}^{n-1}[f(Q- \varepsilon t \mathbf{e})-f(Q)]\right|^{p}r^{n-1}dr \right)^{\frac{1 }{p}} dt+
 $$
 $$+\left(\int\limits_{\omega_{\varepsilon} }d\chi\int\limits_{0}^{ d}
  \left|f(Q)  [1-(1-\varepsilon t/r)_{+}^{n-1}]\right|^{p}r^{n-1}dr \right)^{\frac{1 }{p}} dt=I_{11}+I_{12}.
\end{equation*}
By virtue of the average continuity property of the functions belonging to   $L_{p}(\Omega),$  we have $\forall t>0:\, I_{11}\rightarrow 0,\;\varepsilon\rightarrow 0.$
Consider $I_{12}$ and  let us define the function
$$
h_{1}(\varepsilon,t,r):=\left|f(Q)\right| \cdot \left|1-(1-\varepsilon t/r)_{+}^{n-1} \right|.
$$ Apparently, the following relations hold  almost everywhere  in $\Omega$
$$
\forall t>0,\,h_{1}(\varepsilon,t,r) \leq |f(Q)|,\;h_{1}(\varepsilon,t,r)\rightarrow 0,\;\varepsilon \rightarrow0.
$$
Applying the Lebesgue  dominated convergence theorem, we get $I_{12}\rightarrow 0,\;\varepsilon \rightarrow0.$
It implies that
\begin{equation}\label{20z}
\forall t>0,\,\lim\limits_{\varepsilon\rightarrow 0} h(\varepsilon,t)=0.
\end{equation}
Taking into account  \eqref{19z}, \eqref{20z} and    applying  the Lebesgue  dominated convergence theorem again, we obtain
$$
 I_{1}\rightarrow 0,\;\;\varepsilon\rightarrow0  .
$$
Consider $I_{2},$ using the Minkowski inequality, we  get
$$
I^{\frac{1 }{p}}_{2} \leq \frac{\sin\alpha\pi}{\pi\varepsilon^{\alpha}}\left( \int\limits_{\omega_{\varepsilon} }d\chi\int\limits_{0}^{\varepsilon }
\left| \int\limits_{0}^{r }\frac{f (Q-t\mathbf{e})}{t^{1-\alpha} }\left(\frac{r-t}{r} \right)^{n-1} dt\right|^{p}r^{n-1}dr\right)^{\frac{1 }{p}}
 +\left(\int\limits_{\omega_{\varepsilon} }d\chi\int\limits_{0}^{\varepsilon}
\left| f (Q) \right|^{p}r^{n-1}dr \right)^{\frac{1 }{p}}=
$$
$$
=I_{21} +I_{22}.
$$
Applying the generalized  Minkowski inequality, we obtain
$$
I_{21}\frac{\pi}{\sin\alpha\pi}= \varepsilon^{-\alpha} \left(\int\limits_{\omega_{\varepsilon} }d\chi
\int\limits_{0}^{\varepsilon }
\left|  \int\limits_{0}^{r }\frac{f (Q-t\mathbf{e})}{t^{1-\alpha} }\left(\frac{r-t}{r} \right)^{n-1} \!\!\! dt\right|^{p}r^{n-1}\! dr \right)^{\frac{1 }{p}}\leq
$$
$$
\leq\varepsilon^{-\alpha}\left\{\int\limits_{\omega_{\varepsilon} }\!\!\left[\int\limits_{0}^{\varepsilon }\!\!t ^{\alpha-1 }\!\!
\left( \int\limits_{t}^{\varepsilon }\!\!|f (Q -t \mathbf{e})|^{p}\!\left(\frac{r-t}{r} \right)^{\!\!\!(p-1)(n-1)}\!\!\!(r-t)^{n-1} \!dr  \right)^{\frac{1 }{p}}\!\!dt\right]^{p}\!\!d\chi \right\}^{\frac{1 }{p}}\leq
 $$
 $$
\leq\varepsilon^{-\alpha}\left\{\int\limits_{\omega_{\varepsilon} }\left[\int\limits_{0}^{\varepsilon } t ^{\alpha-1 }
\left( \int\limits_{t}^{\varepsilon } \left|f (P+[r-t]\mathbf{e})\right|^{p}  (r-t)^{n-1}dr  \right)^{\frac{1 }{p}}\!\!dt\right]^{p}\!\!d\chi \right\}^{\frac{1 }{p}}\leq
 $$
  $$
\leq\varepsilon^{-\alpha}\left\{\int\limits_{\omega_{\varepsilon} }\left[\int\limits_{0}^{\varepsilon } t ^{\alpha-1 }
\left( \int\limits_{0}^{\varepsilon } |f (P +r \mathbf{e})|^{p}  r^{n-1}dr  \right)^{\frac{1 }{p}}\!\!dt\right]^{p}\!\!d\chi \right\}^{\frac{1 }{p}}\!\!= \alpha^{-1} \| f\| _{L_{p}( \Delta_{\varepsilon})},
$$
where
$
\Delta_{\varepsilon} :=\{Q\in\Omega_{\varepsilon},\,r<\varepsilon \} .
 $
Note that   ${\rm mess}\, \Delta_{\varepsilon}\rightarrow 0,\,\varepsilon\rightarrow0,$
therefore $I_{21},I_{22}\rightarrow 0,\,\varepsilon\rightarrow0 .$ It follows  that $I_{2 }\rightarrow 0,\,\varepsilon\rightarrow0.$ In the same way, we obtain   $I_{3 }\rightarrow 0,\,\varepsilon\rightarrow 0.$ Since we proved that $I_{1},I_{2},I_{3}\rightarrow 0,\,\varepsilon\rightarrow0,$ then relation \eqref{10.01} holds.
 This completes the proof corresponding to the left-side case.
 The proof corresponding to the right-side case is absolutely analogous.
\end{proof}

The  following theorem  proved in \cite{firstab_lit:1kukushkin2018} establishes the strictly accretive  property   (see \cite{firstab_lit:kato1980})  of the    Kipriyanov    operator what gives us an opportunity to establish   the numerical range of values of the operator, the latter  notion  plays a significant role in the spectral theory.   Denote by   ${\rm Lip}\, \lambda,\; 0<\lambda \leq1  $  the set of functions satisfying the Holder-Lipschitz condition
$$
{\rm Lip}\, \lambda:=\left\{\rho(Q):\;|\rho(Q)-\rho(P)|\leq M r^{\lambda},\,P,Q\in \bar{\Omega}\right\}.
$$

\begin{teo}\label{T5}
  Suppose  $\rho(Q)$ is a real non-negative function,   $\rho\in{\rm Lip}\, \lambda,\; \lambda>\alpha;$
  then the following inequality    holds
\begin{equation}\label{28}
 {\rm Re} ( f,\mathfrak{D}^{\alpha}f)_{L_2(\Omega,\rho)}\geq   C_{\alpha,\rho}    \|f\|^{2}_{L_2(\Omega,\rho)},\;f\in H^{1}_{0} (\Omega),
\end{equation}
where
$$
C_{\alpha,\rho}=\frac{1}{2   \mathrm{d}^{\alpha}}  \left\{  \frac{1}{\Gamma (1-\alpha)} +\frac{(n-1)!}{\Gamma(n-\alpha) }-
\frac{\alpha M   \mathrm{d} ^{\lambda}  }{2\Gamma(1-\alpha)(\lambda-\alpha)\inf  \rho}\right\}.
$$
Moreover, if  we have in additional that for every fixed direction $\mathbf{e}$ the function  $\rho$ is   monotonically  non-increasing,   then
$$
C_{\alpha,\rho}=\frac{1}{2\mathrm{d} ^{ \alpha}}  \left\{  \frac{1}{\Gamma (1-\alpha)} +\frac{(n-1)!}{\Gamma(n-\alpha) }\right\}.
$$
\end{teo}

Consider a linear combination of the uniformly elliptic operator, which is written in the divergence form, and
  a composition of the   fractional integro-differential  operator, where the fractional  differential operator is understood as the adjoint  operator  regarding  the Kipriyanov operator  (see  \cite{firstab_lit:kipriyanov1960},\cite{firstab_lit:1kipriyanov1960},\cite{kukushkin2019})
\begin{equation*}
 L  :=-  \mathcal{T}  \, +\mathfrak{I}^{\sigma}_{ 0+}\rho\, \mathfrak{D}  ^{ \alpha }_{d-},
\; \sigma\in[0,1) ,
 $$
 $$
   \mathrm{D}( L )  =H^{2}(\Omega)\cap H^{1}_{0}( \Omega ),
  \end{equation*}
where
$\,\mathcal{T}:=D_{j} ( a^{ij} D_{i}\cdot),\,i,j=1,2,...,n,$
under    the following  assumptions regarding        coefficients
\begin{equation} \label{16}
     a^{ij}(Q) \in C^{2}(\bar{\Omega}),\, \mathrm{Re} a^{ij}\xi _{i}  \xi _{j}  \geq   \gamma_{a}  |\xi|^{2} ,\,  \gamma_{a}  >0,\,\mathrm{Im }a^{ij}=0 \;(n\geq2),\,
 \rho\in L_{\infty}(\Omega).
\end{equation}
Note that in the one-dimensional case the operator $\mathfrak{I}^{\sigma }_{ 0+} \rho\, \mathfrak{D}  ^{ \alpha }_{d-}$ is reduced to   a  weighted fractional integro-differential operator  composition, which was studied properly  by many researchers (see introduction, \cite[p.175]{firstab_lit:samko1987}).\\

\noindent{\bf 2.   The semi-group model  }\\

Bellow, we  explore  a special  operator class for which    a number of  spectral theory theorems can be applied. Further we construct an abstract  model of a  differential operator in terms of m-accretive operators and call it an m-accretive operator transform, we  find  such conditions that    being  imposed guaranty  that the transform   belongs to the class.  As an application of the obtained abstract results  we study a differential  operator   with a fractional  integro-differential operator   composition  in  final terms on a bounded domain of the $n$ -  dimensional Euclidean space. One of the central points is a relation connecting fractional powers of m-accretive operators and fractional derivative in the most general sense. By virtue of such an approach we express fractional derivatives in terms of   infinitesimal generators, in this regard   the Kipriyanov operator is considered.

We represent propositions devoted to properties of  accretive operators and related questions.
For the reader convenience, we would like to establish well-known facts  of the operator theory under  an appropriate   point of view.
\begin{lem}\label{L1}
Assume that   $A$ is  a  closed densely defined  operator, the following    condition  holds
\begin{equation}\label{5}
\|(A+\lambda)^{-1}\|_{\mathrm{R} \rightarrow \mathfrak{H}}\leq\frac{1}{ \lambda},\,\lambda>0,
\end{equation}
where a notation  $\mathrm{R}:=\mathrm{R}(A+\lambda)$ is used. Then the operators  $A,A^{\ast}$ are m-accretive.
\end{lem}

In accordance with the definition given  in \cite{firstab_lit:Krasnoselskii M.A.}, we can define  a positive and negative  fractional powers of a positive  operator $A$ as follows
\begin{equation}\label{9}
A^{\alpha}:=\frac{\sin\alpha \pi}{\pi}\int\limits_{0}^{\infty}\lambda^{\alpha-1}(\lambda  +A)^{-1} A \,d \lambda;\,\,A^{-\alpha}:=\frac{\sin\alpha \pi}{\pi}\int\limits_{0}^{\infty}\lambda^{-\alpha}(\lambda  +A)^{-1}  \,d \lambda,\,\alpha\in (0,1).
\end{equation}
This definition can be correctly extended  on m-accretive operators, the corresponding reasonings can be found in \cite{firstab_lit:kato1980}. Thus, further we define positive and negative fractional powers of m-accretive operators by formula \eqref{9}. The following lemma reflects the property of the fractional powers of m-accretive operators what gives us the  invaluable technique  to deal  with the infinitesimal generators.

 \begin{lem}\label{L3.1}
 Assume that $\alpha\in(0,1),$ the operator   $J$ is m-accretive,    $  J^{-1} $ is bounded, then
\begin{equation}\label{15.09}
\|J^{-\alpha}f\|_{\mathfrak{H}}\leq C_{   1-\alpha}\|f\|_{\mathfrak{H}},\,C_{ 1-\alpha }= 2 (1-\alpha)^{-1} \|J^{-1}\|  + \alpha^{-1},\,f\in \mathfrak{H}.
\end{equation}
\end{lem}

Consider a transform of an m-accretive operator $J$ acting in $\mathfrak{H}$
\begin{equation}\label{12.0.1}
 Z^{\alpha}_{G,F}(J):= J^{\ast}GJ+FJ^{\alpha},\,\alpha\in [0,1),
\end{equation}
where symbols  $G,F$  denote  operators acting in $\mathfrak{H}.$    Further, using a relation $L= Z^{\alpha}_{G,F}(J)$ we mean that there exists an appropriate representation for the operator $L.$
\noindent The following theorem gives us a tool to describe spectral properties  of transform  \eqref{12.0.1},
  as it will be shown   further  it has an important  application in fractional calculus  since  allows   to represent fractional differential  operators as a transform of the infinitesimal  generator of a    semigroup.

\begin{teo}\label{T1}
 Assume that  the operator   $J$ is m-accretive,    $  J^{-1} $ is compact, $G$ is bounded, strictly accretive, with  a lower bound $\gamma_{G}> C_{ \alpha} \|J^{-1}\|\cdot \|F\|,\; \mathrm{D}(G)\supset \mathrm{R}(J),$   $F\in \mathcal{B}(\mathfrak{H}),$ where $C_{\alpha}$  is a   constant   \eqref{15.09}.
 Then   $Z^{\alpha}_{G,F}(J)$ satisfies  conditions  H1 -  H2.

\end{teo}

Consider the shift semigroup in a direction acting on $L_{2}(\Omega)$ and  defined as follows
$
T_{t}f(Q)=f(Q+\mathbf{e}t),
$
where $Q\in \Omega,\,Q=P+\mathbf{e}r.$
 Bellow, we represent the  complete proof of the lemma proved in \cite{kukushkin2021a} to show the reader some techniques related to the shift semigroup.

\begin{lem}\label{L4}
The semigroup $T_{t}$ is a $C_{0}$  semigroup of contractions.
\end{lem}
 \begin{proof}
     By virtue of the continuous in average property, we conclude that $T_{t}$ is a strongly continuous semigroup. It can be easily established  due to the following reasonings, using the Minkowski inequality, we have
 $$
 \left\{\int\limits_{\Omega}|f(Q+\mathbf{e}t)-f(Q)|^{2}dQ\right\}^{\frac{1}{2}}\leq  \left\{\int\limits_{\Omega}|f(Q+\mathbf{e}t)-f_{m}(Q+\mathbf{e}t)|^{2}dQ\right\}^{\frac{1}{2}}+
 $$
 $$
 +\left\{\int\limits_{\Omega}|f(Q)-f_{m}(Q)|^{2}dQ\right\}^{\frac{1}{2}}+\left\{\int\limits_{\Omega}|f_{m}(Q)-f_{m}(Q+\mathbf{e}t)|^{2}dQ\right\}^{\frac{1}{2}}=
 $$
 $$
 =I_{1}+I_{2}+I_{3}<\varepsilon,
 $$
where $f\in L_{2}(\Omega),\,\left\{f_{n}\right\}_{1}^{\infty}\subset C_{0}^{\infty}(\Omega);$  $m$ is chosen so that $I_{1},I_{2}< \varepsilon/3 $ and $t$
is chosen so that $I_{3}< \varepsilon/3.$
Thus,  there exists such a positive  number $t_{0}$  that
$$
\|T_{t}f-f\|_{L_{2} }<\varepsilon,\,t<t_{0},
$$
for arbitrary small $\varepsilon>0.$  Using the assumption that  all functions have the zero extension   outside $\bar{\Omega},$    we have
$\|T_{t}\|  \leq 1.$ Hence  we conclude that $T_{t}$ is a $C_{0}$ semigroup of contractions (see \cite{Pasy}).
\end{proof}

The following theorem represented in \cite{firstab_lit:1kukushkin2018}  is formulated in terms of the infinitesimal  generator $-A$ of the semigroup $T_{t}.$   It is a central point in the  application of the spectral theory methods to the abstract  integro-differential constructions.

\begin{teo}\label{T3} We claim that  $L=Z^{\alpha}_{G,F}(A).$ Moreover  if  $ \gamma_{a} $ is sufficiently large in comparison  with $\|\rho\|_{L_{\infty}},$ then $L$ satisfies conditions H1-H2, where we put $\mathfrak{M}:=C_{0}^{\infty}(\Omega),$   if we additionally assume that $\rho \in \mathrm{Lip}\lambda,\,   \lambda>\alpha  ,$ then    $ \tilde{\mathcal{H}}=H.$
\end{teo}
The meaning of the following lemma is rather significant since  it  establishes  a very useful  property of  the infinitesimal generator $-A$ of the semigroup $T_{t},$ using which  we can construct a  Hilbert space corresponding to the operator $A$ let alone the secondary fact establishing  the core of the operator $A.$
\begin{lem}\label{L2} We claim that
$
A=\tilde{A_{0}},\,\mathrm{N}(A)=0,
$
where $A_{0}$  is a restriction of $A$ on the set
   $ C^{\infty}_{0}( \Omega ).$
 \end{lem}

In the  following paragraph, we study generalized constructions originated from the shift semigroup, they   may be also interesting due to the applications related to the multidimensional case as well as being  themselves non-standard constructions demonstrating one more class for which hypotheses H1,H2 hold.\\

\noindent{\bf 3.  Further generalizations}\\

Consider a linear  space
$
\mathbb{L}^{n}_{2}(\Omega):=\left\{f=(f_{1},f_{2},...,f_{n}),\,f_{i}\in L_{2}(\Omega)\right\},
$
endowed with the inner product
$$
(f,g)_{\mathbb{L}^{n}_{2}}=\int\limits_{\Omega} (f, g)_{\mathbb{E}^{n}} dQ,\,f,g\in \mathbb{L}^{n}_{2}(\Omega).
$$
It is clear that this  pair forms a Hilbert space and let us use the same  notation $\mathbb{L}^{n}_{2}(\Omega)$ for it.
Consider a    sesquilinear  form
$$
t(f,g):=\sum\limits_{i=1}^{n}\int\limits_{\Omega} (f ,\mathbf{e_{i}})_{\mathbb{E}^{n}}\overline{(g,\mathbf{e_{i}})}_{\mathbb{E}^{n}} dQ,\,f,g\in \mathbb{L} ^{n}_{2} (\Omega),
$$
where $\mathbf{e_{i}}$   corresponds to  $P_{i}\in \partial\Omega,\,\,i=1,2,...,n$   (i.e. $Q=P_{i}+ \mathbf{e_{i}}r     $). The proofs of the   propositions represented in this paragraph   are given  in \cite{firstab_lit(norm eq)}.
\begin{lem}\label{L3}
The points $P_{i}\in \partial\Omega,\,i=1,2,...,n$ can be chosen so that the   form $t$ generates an inner product.
\end{lem}

Consider     a pre Hilbert  space $ \mathbf{L} ^{n}_{2}(\Omega):=\{f:\,f\in \mathbb{L}^{n}_{2}(\Omega)\}$ endowed with the inner product
$$
(f,g)_{\mathbf{L} ^{n}_{2} }:=\sum\limits_{i=1}^{n}\int\limits_{\Omega} (f ,\mathbf{e_{i}})_{\mathbb{E}^{n}}\overline{(g,\mathbf{e_{i}})}_{\mathbb{E}^{n}} dQ,\,f,g\in \mathbb{L} ^{n}_{2} (\Omega),
$$
where $\mathbf{e_{i}}$   corresponds to  $P_{i}\in \partial\Omega,\,\,i=1,2,...,n\,,$  the following     condition  holds
\begin{equation*} \Delta= \left|
\begin{array}{cccc}
P_{11}&P_{12}&...&P_{1n}\\
P_{21}&P_{22}&...&P_{2n}\\
...&...&...&...\\
P_{n1}&P_{n2}&...&P_{nn}
\end{array}
\right|\neq0,
\end{equation*}
where $P_{i}=(P_{i1},P_{i2},...,P_{in}).$
The following theorem establishes a norm equivalence.
\begin{teo}\label{T1}
The   norms $\|\cdot\|_{ \mathbb{L}^{n}_{2} }$ and $\|\cdot\|_{\mathbf{L} ^{n}_{2} } $ are equivalent.
\end{teo}

Consider a pre Hilbert  space
$$
\mathfrak{\widetilde{H}}^{n}_{ A }:= \left \{f,g\in C_{0}^{\infty}(\Omega),\,(f,g)_{\mathfrak{\widetilde{H}}^{n}_{ A }}=\sum\limits_{i=1}^{n}(A_{i} f,A_{i} g)_{L_{2} } \right\},
$$
where $-A_{i}$ is the infinitesimal generator corresponding to the point $P_{i}.$ Here, we should point out that the form $(\cdot,\cdot)_{\mathfrak{\widetilde{H}}^{n}_{ A }} $ generates an inner product due to the fact $\mathrm{N}(A_{i})=0,\,i=1,2,...,n$ proved in Lemma \ref{L2}.
  Let us denote a corresponding Hilbert space by $\mathfrak{H}^{n}_{A}.$
\begin{corol}\label{C1}
  The norms $\|\cdot\|_{\mathfrak{H}^{n}_{A}}$ and $\|\cdot\|_{H_{0}^{1}} $ are equivalent, we have a bounded compact  embedding
$$
\mathfrak{H}^{n}_{A}\subset\subset L_{2}(\Omega).
$$
\end{corol}

Bellow, we aim to represent  an  operator in terms of the infinitesimal generator of the  shift semigroup in a direction with the purpose to    apply the    results \cite{firstab_lit(arXiv non-self)kukushkin2018}, \cite{firstab_lit:1kukushkin2021},  \cite{firstab_lit(frac2023)} to the established representation. In this way we come to natural conditions in terms of the infinitesimal generator of the  shift semigroup in a direction what gives us the desired  result represented in  \cite{kukushkin2021a}. The following theorem allows us to express  the construction  of the partial differential operator in terms of the semigroup theory (having chosen the shift semigroup in the direction)  what reveals a mathematical nature of the operator $-\mathcal{T}.$
\begin{teo}\label{T2}
We claim that
\begin{equation}\label{10}
 -\mathcal{T}=\frac{1}{n}\sum\limits^{n}_{i=1} A_{i}^{\ast}G_{i}A_{i},
\end{equation}
the following relations hold
$$
-\mathrm{Re}(\mathcal{T}f,f)_{L_{2}}\geq C  \|f\|_{\mathfrak{H}^{n}_{A}};\, |(\mathcal{T}f,g)_{L_{2}}|\leq C \|f\|_{\mathfrak{H}^{n}_{A}}\|g\|_{\mathfrak{H}^{n}_{A}},\;f,g\in C_{0}^{\infty}(\Omega),
$$
where $G_{i}$ are some operators corresponding to the operators $A_{i}.$
\end{teo}

Thus, by virtue of     Corollary \ref{C1} and Theorem \ref{T2}, we are able to   claim that  hypotheses H1,H2   \cite{kukushkin2021a} hold  for    the operator $-\mathcal{T}.$ It is rather reasonable to represent analog of Theorem \ref{T3} which reflects connection between the operator $-\mathcal{T}$ and its perturbation by the Kipriyanov operator.

\begin{teo}\label{T6}
 We claim that
\begin{equation}\label{19}
L=\frac{1}{n}\sum\limits^{n}_{i=1} A_{i}^{\ast}G_{i}A_{i}+FA_{1}^{\alpha},
\end{equation}
where $F $ is a bounded operator, $P_{1}:=P,$ and $G_{i}$ are the same as in Theorem \ref{T2}. Moreover  if  $ \gamma_{a} $ is sufficiently large in comparison  with $\|\rho\|_{L_{\infty}},$ then
the following relations hold
$$
 \mathrm{Re}(Lf,f)_{L_{2}}\geq C  \|f\|_{\mathfrak{H}^{n}_{A}};\, |(Lf,g)_{L_{2}}|\leq C \|f\|_{\mathfrak{H}^{n}_{A}}\|g\|_{\mathfrak{H}^{n}_{A}},\;f,g\in C_{0}^{\infty}(\Omega).
$$
\end{teo}
  The theorem reveals a remarkable fact the  perturbation  preserves the property being in the class satisfying hypotheses H1,H2 what makes the perturbed operator interesting itself from the theoretical point of view let alone a prospective  applications determined by    convenience, from the technical point of view, in dealing with    the invented   operator construction in the multidimensional space.

\section{Integro-differential constructions}

\noindent{\bf 1.  Abel-Lidskii root vectors series expansion}\\

In this section, we represent a theorem   valuable from  theoretical  and applied points.   It is based upon the modification  of the Lidskii method,  this  is why  following the the classical approach we divided it into three  statements that can be claimed separately. The first  statement  (Theorem 3  \cite{firstab_lit(frac2023)}) establishes a character of the series convergence having a principal meaning within the whole concept. The second statement (Theorem 3  \cite{firstab_lit(frac2023)}) reflects the name of convergence - Abel-Lidskii since  the latter   can be connected with the definition of the series convergence in the Abel sense, more detailed  information can be found in the monograph by Hardy G.H. \cite{firstab_lit:Hardy}. The third  statement (Theorem 4 \cite{firstab_lit(frac2023)})  is a valuable  application of the first one, it is based upon   suitable algebraic reasonings having been noticed by the author and allowing   to involve   a fractional derivative in the first term. We should note that previously,  a concept of an operator function represented in the second term  was realized  in the paper  \cite{firstab_lit:2kukushkin2022}, where  a case corresponding to  a function   represented by a Laurent series with a polynomial regular part was considered.  Bellow, we consider  a comparatively more difficult  case obviously  related to the infinite regular part of the Laurent series and therefore  requiring  a principally different method of study.

  It is a well-known fact that each  eigenvalue  $\mu_{q},\,q\in \mathbb{N}$ of the compact operator  $B$ generates a  set  of  Jordan chains containing  eigenvectors and  root vectors. Denote by $m(q)$   a geometrical multiplicity of the corresponding  eigenvalue and   consider a Jordan chain corresponding to an eigenvector $e_{q_{\xi}},\;\xi=1,2,...,m(q),$   we have
 \begin{equation}\label{12i}
  e_{q_{\xi}},e_{q_{\xi}+1},...,e_{q_{\xi}+k(q_{\xi})},
\end{equation}
where $k(q_{\xi})$ indicates a number   of elements in the    Jordan chain, the symbols except for the first one denote root vectors of the operator $B.$ Note that  combining the Jordan chains corresponding to an eigenvalue, we obtain  a Jordan basis in the invariant subspace generated by the eigenvalue, moreover  we can arrange  a so-called system of major vectors $\{e_{i}\}_{1}^{\infty}$ (see \cite{firstab_lit:1Lidskii}) of the operator $B$  having combined Jordan chains. It is remarkable that  the eigenvalue $\bar{ \mu}_{q} $ of the operator $B^{\ast}$   generates the Jordan chains of the operator $B^{\ast}$ corresponding to  \eqref{12i}. In accordance with \cite{firstab_lit:1kukushkin2021}, we have
$$
 g_{q_{\xi}+k(q_{\xi})} ,\; g_{q_{\xi}+k(q_{\xi})-1} ,...,
 g_{q_{\xi}},
$$
where the symbols except for the first one denote root vectors of the operator $B^{\ast}.$ Combining Jordan chains of the operator $B^{\ast},$    we can construct  a biorthogonal system $\{g_{n}\}_{1}^{\infty}$ with respect to the system of the major vectors of the operator $B.$ This  fact is given in detail in the paper \cite{firstab_lit:1kukushkin2021}.
The following construction plays a significant role in the theory created in the papers  \cite{firstab_lit:1kukushkin2021},\cite{firstab_lit:2kukushkin2022},\cite{firstab_lit(axi2022)}   and therefore  deserves to be considered  separately,  denote
\begin{equation}\label{3h}
 \mathcal{A} _{\nu}(\varphi,t)f:= \sum\limits_{q=N_{\nu}+1}^{N_{\nu+1}}\sum\limits_{\xi=1}^{m(q)}\sum\limits_{i=0}^{k(q_{\xi})}e_{q_{\xi}+i}c_{q_{\xi}+i}(t),
\end{equation}
where $\{N_{\nu}\}_{1}^{\infty}$ is  a sequence of natural numbers,
\begin{equation}\label{4h}
c_{q_{\xi}+i}(t)=   e^{ -\varphi(\lambda_{q})  t}\sum\limits_{j=0}^{k(q_{\xi})-i}H_{j}(\varphi, \lambda_{q},t)c_{q_{\xi}+i+j},\,i=0,1,2,...,k(q_{\xi}),
\end{equation}
$
c_{q_{\xi}+i}= (f,g_{q_{\xi}+k-i}) /(e_{q_{\xi}+i},g_{q_{\xi}+k-i}),
$
$\lambda_{q}=1/\mu_{q}$ is a characteristic number corresponding to $e_{q_{\xi}},$
$$
H_{j}( \varphi,z,t ):=  \frac{e^{ \varphi(z)  t}}{j!} \cdot\lim\limits_{\zeta\rightarrow 1/z }\frac{d^{j}}{d\zeta^{\,j}}\left\{ e^{-\varphi (\zeta^{-1})t}\right\} ,\;j=0,1,2,...\,,\,.
$$
More detailed information on the considered above   Jordan chains can be found in \cite{firstab_lit:1kukushkin2021}.\\

\noindent{\bf 2.  Decomposition theorem}\\

Denote by $\mathfrak{H}$   the abstract separable Hilbert space and assume that the hypotheses H1,H2 hold for the operator $W$ acting in $\mathfrak{H}.$  We should point out that such  chose of the operator class  justified by both abstract theoretical relevance related to the spectral properties of non-selfadjoint operators and the concrete applications including ones involving the Kipriyanov operator.
   Denote by
\begin{equation}\label{12a}
\varphi(W):=\sum\limits_{n=l}^{k}  c_{n}  W^{n},\;\;-\infty\leq l,k\leq\infty
\end{equation}
a formal construction called by an operator function, where $c_{n}$ are the coefficients corresponding to the  function of the complex variable  $\varphi.$
Here, we ought  to make a bibliographic digression and remind that the case $l=-\infty,\,k<\infty$ was considered in \cite{firstab_lit(axi2022)}. In this case,  the complex function $\varphi$ was supposed to have a decomposition into the Laurent series about the point zero with the coefficients  $c_{n}$ satisfying the additional assumption
\begin{equation}\label{eq16s}
\max\limits_{n=0,1,...,k}(|\mathrm{arg} c_{n}|+n\theta)<\pi/2,
\end{equation}
 where $\theta$ is the semi-angle of the sector containing the numerical range of values of the operator $W.$ We should note that the problem connected with the representation  \eqref{12a} can be divided on two parts $l\geq-\infty,\,k=0$ and $l=0,\,k\leq\infty,$ thus the    first one was properly studied in \cite{firstab_lit(axi2022)}, the second one was studied in \cite{firstab_lit(frac2023)}, with the following assumptions (we represent a technical variant, the expended variant can be found in \cite{firstab_lit(frac2023)}):    the complex function $\varphi$ of the order less than a half    maps  the ray  $\mathrm{arg}\,z=\theta_{0}$   within  a sector $\mathfrak{L}_{0}(\zeta),\,0<\zeta<\pi/2,$ the condition holds
\begin{equation}\label{eq17s}
\mathrm{Re} \varphi(z)>   Ce^{ H(\theta_{0})r^{\varrho }}, \mathrm{arg}\,z=\theta_{0},
\end{equation}
where  $H(\theta_{0})$ is a positive number in accordance with the Lemma 1 \cite{firstab_lit(frac2023)}. Taking into account  the above, we can consider a function represented by a Laurent series with the arbitrary principal part and the regular part satisfying \eqref{eq16s},\eqref{eq17s} respectively to the finite, infinite cases. This statement can be   proved by repetition of   the reasonings represented in Lemma 5 \cite{firstab_lit(axi2022)}, thus we  leave the proof  to the reader.

Bellow,   we   consider a Hilbert space  consists of   element-functions $u:\mathbb{R}_{+}\rightarrow \mathfrak{H},\,u:=u(t),\,t\geq0,$      we understand the differentiation and integration operations  in the generalized sense, i.e. the derivative is defined as a  limit  in the sense of the norm e.t.c. (see \cite{firstab_lit:1kukushkin2021}, \cite{firstab_lit:Krasnoselskii M.A.}).  Combining the operations, we can define  a generalized  fractional derivative in the Riemann-Liouville sense (see \cite{firstab_lit:2kukushkin2022},\cite{firstab_lit:samko1987}),     in the formal form, we have
$$
   \mathfrak{D}^{1/\alpha}_{-}f(t):=-\frac{1}{\Gamma(1-1/\alpha)}\frac{d}{d t}\int\limits_{0}^{\infty}f(t+x)x^{-1/\alpha}dx,\;\alpha\geq1,
$$
here we should note that     facts  $\mathfrak{D}^{1}_{-}f(t)=-du/dt,\,\mathfrak{D}^{0}_{-}f(t)=f(t)$ can be obtained due to the definition of the operator (see \cite{firstab_lit:samko1987}).
In terms of the expression \eqref{12a}, consider    a Cauchy problem
\begin{equation}\label{23a}
     \mathfrak{D}^{1/\alpha}_{-}  u = \varphi(W) u  ,\;u(0)=f\in \mathrm{D}(W^{n} ),\,  n=1,2,...k.
\end{equation}
  Taking into account the above, combining   results \cite{firstab_lit(axi2022)}, \cite{firstab_lit(frac2023)}, we can formulate the following theorem.
\begin{teo}\label{T7}
 Assume that conditions   \eqref{eq16s},\eqref{eq17s} hold   respectively to the cases corresponding to the finite, infinite regular part of the series \eqref{12a},       then
there exists a solution of the Cauchy problem \eqref{23a} in the form
\begin{equation}\label{25k}
u(t)= \sum\limits_{\nu=0}^{\infty} \mathcal{A}_{\nu}(\varphi^{\alpha},t)f,\;  \sum\limits_{\nu=0}^{\infty}\| \mathcal{A}_{\nu}(\varphi^{\alpha},t)f\|<\infty.
\end{equation}
Moreover, the existing solution is unique if the operator   $\mathfrak{D}^{1-1/\alpha}_{-}\!\varphi(W)$ is accretive.
 \end{teo}
\begin{proof} To avoid  any kind of repetition, we does not represent the complete proof  having restricted reasonings  by the scheme appealing to     Lemma 5, Theorem 1 \cite{firstab_lit(axi2022)},  Lemma 3, Theorem 4 \cite{firstab_lit(frac2023)}. Thus, the detailed calculation is left to the reader.
\end{proof}
Further, considering an operator function, we will assume that  conditions \eqref{eq16s},\eqref{eq17s} hold respectively to the case.
Now consider transform \eqref{12.0.1}
\begin{equation*}
 Z^{\alpha}_{G,F}(J)= J^{\ast}GJ+FJ^{\alpha},\,\alpha\in [0,1),
\end{equation*}
 assuming that the conditions of   Theorem \ref{T1} hold,   we can consider a Cauchy problem involving the transform which represents an  integro-differential construction  in the generalized sense. The latter problem appeals to a plenty of concrete evolution equations which form a base for  the modern engineering  sciences. Let us contemplate then a magnificent representative of the generators creating  the transform - the directional derivative  which fractional power is the Kipriyanov operator. Consider a Cauchy problem
\begin{equation}\label{eq20v}
     \mathfrak{D}^{1/\alpha}_{-}  u  =\sum\limits_{n=l}^{k}  c_{n} L^{n} u  ,\;u(0)=f\in C_{0}^{\infty}(\Omega) ,\,  n=1,2,...k,
\end{equation}
where we are dealing with the following integro-differential construction with the sufficiently smooth coefficients for which the conditions \eqref{16} hold
$$
L  :=-  \mathcal{T}  \, +\mathfrak{I}^{\sigma}_{ 0+}\rho\, \mathfrak{D}  ^{ \gamma }_{d-},
\; \sigma,\gamma\in[0,1).
$$
Thus, we can easily see that the case $\sigma=0,\,\gamma=0,\,-\infty<l,k<\infty$  leads to the class of integro-differential equations of the integer order and the obtained results give us a method to solve the corresponding Cauchy problems \eqref{eq20v}. Certainly, we can consider a closure of the defined  operator function   on the set $C_{0}^{\infty}(\Omega),$ the fact that it admits closure is proved in \cite{firstab_lit KRAUNC}. The case corresponding to arbitrary values of $\sigma,\gamma$ within the range, requires more peculiar technique that was considered in paragraph 3.1 \cite{firstab_lit:2kukushkin2022}. However, Theorem \ref{T7} becomes relevant in solving evolution equations with an integro-differential operator in the second term to say nothing on the far reaching generalizations corresponding to an operator function with the infinite principal or regular part of its   Laurent series.

\section{Conclusions}

In the paper, there was represented a historical  survey devoted to the achievements of  Kipriyanov I.A. where the exclusively  constructed fractional calculus theory was discussed  convexly. We produced  a comparison analysis in the framework of ways and means related to  further prospective generalizations  of   fractional derivative as a notion.  The qualitative properties of the Kipriyanov fractional differential operator were studied by the methods of the classical fractional calculus theory in contrast to the exclusive approach invented by Kipriyanov I.A. Having taken as a basis the concept of multidimensional generalization of the fractional differential operator
in the sense of Marchaud, we adapted  the previously known technique of the proofs related to  the theory
of fractional calculus of one variable. Along with the previously known definition of a fractional derivative introduced by  Kipriyanov I.A. we used
  a new definition of a multidimensional fractional integral in the direction  what
allows to describe the range of the Kipriyanov adjoint operator. A number of statements having analogues in the theory of fractional calculus of one variable and previously proved by the author  were discussed. In particular
  the classical result - a sufficient condition  for representability by a fractional directional integral in the direction was observed.
The strictly accretive property of the Kipriyanov operator, being an outstanding author's result,    was observed. On the base of the given technique, there were developed methods of the semigroup theory and the spectral theory of non-selfadjoint operators what leads us to significant applications to the abstract evolution equations in the Hilbert space. The latter gives us an opportunity to solve a whole class of problems related to integro-differential equations of the real order wherein the  one related to the semigroup connected with the Kipriyanov operator was studied properly.

\end{document}